\documentclass[12pt]{amsart}
\usepackage{amssymb}
\usepackage{hyperref}

\newtheorem{thm}{Theorem}[section]
\newtheorem{lem}[thm]{Lemma}
\newtheorem{cor}[thm]{Corollary}
\newtheorem{prop}[thm]{Proposition}
\newtheorem{ex}[thm]{Example}

\newtheorem{qu}{Question}

\theoremstyle{definition}

\newtheorem{defi}[thm]{Definition}

\theoremstyle{remark}

\newtheorem{rem}[thm]{Remark}
\newtheorem*{rem*}{Remark}

\DeclareMathOperator{\Aut}{Aut}

\newcommand{\kringel}{\mathbin{\raise1pt\hbox{$\scriptstyle\circ$}}}
\newcommand{\pkt}{\mathbin{\raise0pt\hbox{$\scriptstyle\bullet$}}}

\newcommand{\C}{\mathbb{C}}

\newcommand{\ad}{\mathop{\rm ad}}

\newcommand{\Der}{\mathop{\rm Der}}

\newcommand{\diag}{\mathop{\rm diag}}

\newcommand{\La}{\mathfrak{a}}

\newcommand{\Lf}{\mathfrak{f}}
\newcommand{\Lg}{\mathfrak{g}}
\newcommand{\Lh}{\mathfrak{h}}

\newcommand{\Ll}{\mathfrak{l}}
\newcommand{\Ln}{\mathfrak{n}}

\newcommand{\Lr}{\mathfrak{r}}
\newcommand{\Ls}{\mathfrak{s}}

\newcommand{\CT}{\mathcal{T}}

\newcommand{\al}{\alpha}
\newcommand{\be}{\beta}

\newcommand{\de}{\delta}

\newcommand{\la}{\lambda}

\newcommand{\ra}{\rightarrow}

\renewcommand{\phi}{\varphi}

\begin{document}


\title[Derivation double Lie algebras]{Derivation double Lie algebras}

\author[D. Burde]{Dietrich Burde}
\address{Fakult\"at f\"ur Mathematik\\
Universit\"at Wien\\
Oskar-Morgenstern-Platz 1\\
1090 Wien \\
Austria}
\email{dietrich.burde@univie.ac.at}

\date{\today}

\subjclass[2000]{Primary 17B60, 17B01}
\keywords{Post-Lie algebra, Derivation Double Lie algebra, Classical R-matrix}

\begin{abstract}
We study classical $R$-matrices $D$ for Lie algebras $\Lg$ such that $D$ is also a derivation of $\Lg$. 
This yields derivation double Lie algebras $(\Lg,D)$.
The motivation comes from recent work on post-Lie algebra structures on pairs of Lie algebras
arising in the study of nil-affine actions of Lie groups. We prove that there are no nontrivial 
simple derivation double Lie algebras, and study related Lie algebra identities for arbitrary
Lie algebras.
\end{abstract}

\maketitle

\section{Introduction}
Let $\Lg$ be a finite-dimensional Lie algebra over a field $K$ of characteristic zero.
Motivated by studies in post-Lie algebras \cite{BU41}, \cite{BU44} we are interested in the following question. 

\begin{qu}\label{qu1}
Let $\Lg$ be a Lie algebra. For which derivations $D$ of $\Lg$ does the skew-symmetric bilinear map
\[
[x,y]_D=D([x,y])
\]
satisfy the Jacobi identity ?
\end{qu}

In other words, for which derivations $D$ defines $[x,y]_D$ another Lie algebra, denoted by $\Lg_D$ ?
If  $[x,y]_D$ is a Lie bracket, then the linear map $D$ is also an example
of a {\it classical $R$-matrix} for $\Lg$, i.e., a linear transformations $R\colon \Lg\ra\Lg$
such that
\[
[x,y]_R=[R(x),y]+[x,R(y)]
\]
defines a Lie bracket. Classical $R$-matrices \cite{SEM} have been studied by many authors. Our main
result here is that for simple Lie algebras $\Lg$ of rank $r\ge 2$ over an algebraically closed field of 
characteristic zero, $[x,y]_D=D([x,y])$ is a Lie bracket 
only for the trivial derivation $D=0$, see Theorem $\ref{3.3}$. On the other hand, for $\Ls\Ll_2(K)$ this is 
a Lie bracket for all derivations $D\in \Der(\Ls\Ll_2(K))$. \\[0.2cm]
Post-Lie algebra structures have been introduced in the context of nil-affine actions of Lie groups
in \cite{BU41}, and also in connection with homology of partition posets and the study of 
Koszul operads in \cite{LOD}, \cite{VAL}. Such structures are important in many
areas of algebra, geometry and physics. They generalize both LR-structures and pre-Lie algebra structures, 
which we have studied in \cite{BU22}, \cite{BU24}, \cite{BU33}, \cite{BU34}, \cite{BU38}. Related topics are Poisson
structures and Lie bialgebra structures, which have been studied as well in connection with classical
$R$-matrices and double Lie algebras, see \cite{AGM}. \\[0.2cm]
Let us explain the motivation for question $\ref{qu1}$ in terms of post-Lie algebra structures.
In \cite{BU44}, Theorem $6.4$ we have determined all complex semisimple Lie algebras $\Lg$ with Lie bracket $[\, ,\,]$, 
and all $\la\in \C$, $z\in \Lg$ such that
\[
[x,y]_R = [z,[x,y]]+(2\la+1)[x,y]
\]
defines a Lie bracket, and at the same time the Lie bracket $[\,  ,\,]$ satisfies
\[
[[z,x],[z,y]]  = [z,[z,[x,y]]]+(2\la+1)[z,[x,y]]+(\la^2+\la)[x,y] 
\]
for all $x,y\in \Lg$. 
It turned out that this always implies $z=0$, except for the case where $\la=-\frac{1}{2}$, and the Lie algebra is 
isomorphic to a direct sum of $\Ls\Ll_2(\C)$'s. For $\la=-\frac{1}{2}$ the second Lie bracket is given
by $[x,y]_D=[z,[x,y]]=D([x,y])$ with the inner derivation $D=\ad (z)$. This motivated the question
for which $z\in \Lg$ the  endomorphism $D=\ad (z)$ is a classical $R$-matrix. \\
Instead of $D=\ad(z)$ we consider here an arbitrary derivation $D$ of any Lie algebra $\Lg$, and study the 
Jacobi identity for the skew-symmetric bilinear map $[x,y]_D=D([x,y])$. \\[0.2cm]
Question $\ref{1}$ naturally leads to another question, which is in particular also interesting
for solvable and nilpotent Lie algebras:

\begin{qu}\label{qu2}
Which Lie algebras $\Lg$ have the property that
$[x,y]_D=D([x,y])$ satisfies the Jacobi identity for all derivations $D\in \Der(\Lg)$ ?
\end{qu}

This is related to the theory of Lie algebra identities, and in particular to the
variety ${\rm var}(\Ls\Ll_2(K))$, see \cite{DRE}, \cite{FIL1}, \cite{FIL2}. 
We study this question together with the related identities \eqref{1}, \eqref{2}, \eqref{3} and \eqref{4}
for some general classes of Lie algebras, and in particular
for all complex nilpotent Lie algebras of dimension $n\le 7$. A Lie algebra has the property given in question 
$\ref{qu2}$ if and only if identity \eqref{1} holds for it, for all derivations.
We show that every almost abelian Lie algebra
satisfies the Hom-Jacobi identity \eqref{2}, and hence also \eqref{1}, see Proposition $\ref{4.6}$.
Finally we prove that every complex CNLA (characteristically nilpotent Lie algebra) of dimension $7$ is a derivation 
double Lie algebra for all derivations, see Proposition $\ref{4.12}$.

\section{Preliminaries}

Let $\Lg$ be a finite-dimensional Lie algebra over a field $K$ of characteristic zero.
Let $\Lg^0=\Lg$, and $\Lg^i=[\Lg,\Lg^{i-1}]$ for all $i\ge 1$. 
We say that $\Lg$ is nilpotent if there exists an index $c\ge 1$ such that $\Lg^c=0$. In that case, the 
smallest such index is called the nilpotency class of $\Lg$ and is denoted by $c(\Lg)$.
Let $\Lg^{(0)}=\Lg$, and $\Lg^{(i)}=[\Lg^{(i-1)},[\Lg^{(i-1)}]$ for all $i\ge 1$.
We say that $\Lg$ is solvable if there exists an index $d\ge 1$ such that $\Lg^{(d)}=0$. In that case, the 
smallest such integer is called the solvability class and is denoted by $d(\Lg)$. \\[0.2cm]
Classical $R$-matrices and double Lie algebras have been defined in \cite{SEM} as follows.

\begin{defi}\label{rm}
Let $V$ be a vector space over a field $K$, and $\Lg=(V, [\, ,])$ be a Lie bracket on $V$.
A linear transformation $R\colon \Lg\ra \Lg$ is called a {\it classical $R$-matrix}, if
\[
[x,y]_R=[R(x),y]+[x,R(y)]
\]
defines a Lie bracket, i.e., satisfies the Jacobi identity. In this case, the pair $(\Lg,R)$ is
called a {\it double Lie algebra}.
\end{defi}

It is useful to set 
\[
B_R(x,y)=[R(x),R(y)]-R([R(x),y]+[x,R(y)]).
\]
Then the Jacobi identity for $[x,y]_R$ can be formulated as follows \cite{SEM}:

\begin{prop}
Let $\Lg$ be a Lie algebra with Lie bracket $[\, ,\,]$.
The bracket $[x,y]_R=[R(x),y]+[x,R(y)]$ satisfies the Jacobi identity if and only if
\[
[B_R(x,y),w]+[B_R(y,w),x]+[B_R(w,x),y] = 0
\]
for all $x,y,w\in \Lg$.
\end{prop}

\begin{defi}
Let $\la\in K$. The identity $B_R(x,y)+\la [x,y]=0$ for all $x,y\in \Lg$ is called MYBE, the 
{\it modified Yang-Baxter equation}. 
\end{defi}

It is obvious that every solution $R$ of MYBE is a classical $R$-matrix. The converse, however, 
need not be true in general. \\[0.2cm]
Concerning question $\ref{qu1}$ we have the following result.

\begin{prop}\label{2.4}
Let $\Lg$ be a Lie algebra and $D\in \Der(\Lg)$ be a derivation. Then $[x,y]_D=D([x,y])
=[D(x),y]+[x,D(y)]$ satisfies the Jacobi identity if and only if
\begin{align}\label{1}
D\Bigl([D(x),[y,w]]+[D(y),[w,x]]+[D(w),[x,y]] \Bigr)=0
\end{align}
for all $x,y,w\in \Lg$.
\end{prop}

\begin{proof}
We have
\begin{align*}
[[x,y]_D,w]_D & = [[D(x),y]+[x,D(y)],w]_D \\
 & = D([[D(x),y],w]) + D([[x,D(y)],w])
\end{align*}
This yields
\begin{align*}
[[x,y]_D,w]_D + [[y,w]_D,x]_D+ [[w,x]_D,y]_D & = D([[D(x),y],w]) + D([[x,D(y)],w]) \\
 & +  D([[D(y),w],x]) + D([[y,D(w)],x]) \\
 & +  D([[D(w),x],y]) + D([[w,D(x)],y]) \\
 & = -D([[y,w],D(x)]+[[w,x],D(y)]+[[x,y],D(w)]).
\end{align*}
In the last step we have used the Jacobi identity three times, i.e., 
\[
[[D(x),y],w]+[[y,w],D(x)]+[[w,D(x)],y]=0,
\]
and similarly for the terms with $D(y)$ and $D(w)$. 
\end{proof}

Note that identity \eqref{1} can also be stated  
as follows: for all $x,y,w\in \Lg$ we have

\begin{align*}
0 & = [x,[D(y),D(w)]]+[y,[D(w),D(x)]]+[w,[D(x),D(y)]] \\
  & + [D^2(y),[x,w]]+[D^2(w),[y,x]]+[D^2(x),[w,y]].
\end{align*}

\begin{defi}
Let $\Lg$ be a Lie algebra and $D\in \Der(\Lg)$ be a derivation, such that $[x,y]_D=D([x,y])$ 
defines another Lie bracket $\Lg_D$. Then the pair $(\Lg,D)$ is called a {\it derivation double Lie algebra}.
\end{defi}

The identity within the brackets of \eqref{1} for a linear map $D\colon \Lg\ra \Lg$ is called the 
{\it Hom-Jacobi identity}, see \cite{MAS} for further references. It says that
\begin{align}\label{2}
[D(x),[y,z]]+[D(y),[z,x]]+[D(z),[x,y]] =0
\end{align}
for all $x,y,z\in \Lg$. 

\begin{cor}
Let $\Lg$ be a Lie algebra with Lie bracket $[\, , \,]$ and $z\in \Lg$. Then
the bracket $[x,y]_D=[z,[x,y]]$ satisfies the Jacobi identity if and only if
\begin{align}\label{3}
[z,[[z,x],[y,w]]]+[z,[[z,y],[w,x]]]+[z,[[z,w],[x,y]]] & =0
\end{align}
for all $x,y,w\in \Lg$.
\end{cor}

\begin{proof}
This follows from Proposition $\ref{2.4}$ with $D=\ad (z)$.
\end{proof}

For $w=z$ the identity \eqref{3} implies, for $2\neq 0$
\begin{align}\label{4}
[z,[[z,x],[z,y]]]=0
\end{align}
for all $x,y\in \Lg$.

\begin{lem}
Let $\Lg$ be a Lie algebra and suppose that $D=\ad (z)$ is a classical $R$-matrix, so that 
$[x,y]_D=[z,[x,y]]$ defines a second Lie bracket. Then $\ad (z)^3$ is a derivation of $\Lg$.
\end{lem}

\begin{proof}
For $D=\ad (z)$ identity \eqref{4} gives
\begin{align*}
0 & = D([D(x),D(y)]) \\
  & = [D^2(x),D(y)]+[D(x),D^2(y)].
\end{align*}
This yields
\begin{align*}
D^3([x,y]) & = [D^3(x),y]+3[D^2(x),D(y)]+3[D(x),D^2(y)]+[x,D^3(y)] \\
  & = [D^3(x),y]+[x,D^3(y)].
\end{align*}
Hence $D^3=\ad (z)^3$ is a derivation.
\end{proof}

Conversely, if $\ad (z)^3$ is a derivation of $\Lg$, and $3\neq 0$, then identity \eqref{4} 
holds for $z$.

\begin{rem}
An element $z$ of a Lie algebra $\Lg$ is called {\it extremal}, if there is
a linear map $f_z\colon \Lg\ra K$ such that
\[
[z,[z,x]]=f_z(x)z
\]
for all $x\in \Lg$. For the study of extremal elements see \cite{COH} and the references therein.
It is a well known result of Premet, that every simple Lie algebra over an algebraically closed field 
of characteristic different from $2$ and $3$ has a nontrivial extremal element.
Note that for every extremal element $z\in \Lg$ we have $\ad (z)^3=0$, so that identity
\eqref{4} holds for all extremal elements $z\in \Lg$ by the above Lemma.
\end{rem}

\section{Simple derivation double Lie algebras}

We will give here an answer to question $\ref{qu1}$ for simple Lie algebras $\Lg$ over an
algebraically closed field of characteristic zero. In terms of classical $R$-matrices, the question is, 
for which $z\in \Lg$ the linear map $R=\ad (z)$ is a classical $R$-matrix. We will show that
for rank one every $\ad(z)$ is a classical $R$-matrix, and that for rank $r\ge 2$ only
the zero transformation is a classical $R$-matrix. In other words, there are no nontrivial
simple derivation double Lie algebras of rank $r\ge 2$. \\
One should remark that simple Lie algebras always admit nontrivial classical $R$-matrices, but not of the 
form $\ad(z)$. An easy example is given by $R=\la I_n$ with the identity matrix $I_n$. 
In general there are much more possibilities, see \cite{BED}. \\
Denote by $\Lr_{3,1}(\C)$ the $3$-dimensional solvable Lie algebra
given by $[e_1,e_2]=e_2$ and $[e_1,e_3]=e_3$, see table $1$.

\begin{prop}
Let $\Lg=\Ls\Ll(2,\C)$. Then $R=\ad(z)$ is a classical $R$-matrix for all $z\in \Lg$.
The Lie algebra $\Lg_R$ is isomorphic to $\Lr_{3,1}(\C)$ for all $z\neq 0$.
\end{prop}

\begin{proof}
The claim follows by a direct computation. Let $(e_1,e_2,e_3)$ be the standard basis of $\Ls\Ll(2,\C)$
with $[e_1,e_2]=e_3$,  $[e_1,e_3]=-2e_1$,  $[e_2,e_3]=2e_2$, and write
$z=z_1e_1+z_2e_2+z_3e_3$. Using $[x,y]_R=[z,[x,y]]$ we have
\begin{align*}
[e_1,e_2]_R & = [z,[e_1,e_2]]=[z,e_3]=-2z_1e_1+2z_2e_2, \\
[e_1,e_3]_R & = [z,[e_1,e_3]]=[z,-2e_1]=-2z_3e_1+2z_2e_3, \\
[e_2,e_3]_R & = [z,[e_2,e_3]]=[z,2e_2]=-2z_3e_2+2z_1e_3, \\
\end{align*} 
and 
\begin{align*}
[e_1,[e_2,e_3]_R]_R & = [e_1,-2z_3e_2+2z_1e_3]_R=-4z_2z_3e_2+4z_1z_2e_3, \\
[e_2,[e_3,e_1]_R]_R & = [e_2,2z_3e_1-2z_2e_3]_R=4z_1z_3e_1-4z_1z_2e_3, \\
[e_3,[e_1,e_2]_R]_R & = [e_3,-2z_1e_1+2z_2e_2]_R= -4z_1z_3e_1+4z_2z_3e_2. \\
\end{align*}

Hence
\begin{align*}
[e_1,[e_2,e_3]_R]_R+[e_2,[e_3,e_1]_R]_R+[e_3,[e_1,e_2]_R]_R  & = 0.
\end{align*}

It is easy to see that the resulting Lie algebra $\Lg_R$ is isomorphic to
$\Lr_{3,1}(\C)$, except for $z=0$.

\end{proof}

\begin{thm}\label{3.3}
Let $\Lg$ be a simple Lie algebra of rank $r\ge 2$ over an algebraically closed field $K$ of characteristic 
zero, and $z\in \Lg$. Suppose that $R=\ad(z)$ is a classical $R$-matrix. Then $z=0$ and $R=0$.
\end{thm}

\begin{proof}
Let $G$ be the identity component of the algebraic group $\Aut(\Lg)$, i.e., $G=\Aut(\Lg)^{\circ}$.
Then $G$ acts on $\Lg$ and the set of $z\in \Lg$ satisfying the identity \eqref{3} is 
a $G$-invariant closed set. Denote by $z=s+n$ the Jordan-Chevalley decomposition of $z$, with
semisimple part $s$ and nilpotent part $n$. We have $\ad (z)=\ad(z)_s+\ad(z)_n=\ad(s)+\ad(n)$, 
since $\Lg$ is simple. Since the identity \eqref{3} holds also for the 
orbit closure and $Gs\subseteq \overline{Gz}$ is closed, we may apply a standard limit argument
and pass to the semisimple part $s$ of $z$. But for semisimple elements $s$ we will show that the identity forces $s=0$.
Hence we obtain that $z=n$ must be nilpotent. \\
Let $z=s\neq 0$ be semisimple. Let $\al,\be $ be roots such that $\al+\be$ is again a root. We may assume that
$(\al+\be)(z)\neq 0$, since $\Lg$ has rank at least $2$. Now we take $(x,y,w)=(h,e_{\al},e_{\be})$ for identity
\eqref{1}, with $h\in \Lh$, the Cartan subalgebra of $\Lg$.
We have $[e_{\al},e_{\be}]=n_{\al\be}e_{\al+\be}$ with $n_{\al\be}\neq 0$, since $\al+\be$ is a root.
Furthermore $[h,e_{\al}]=\al(h)e_{\al}$ and $[h,e_{\be}]=\be(h)e_{\be}$. We have
\begin{align*}
[[z,x],[y,w]] + [[z,y],[w,x]] + [[z,w],[x,y]] & = [[z,h],[e_{\al},e_{\be}]]+[[z,e_{\al}],[e_{\be},h]]+[[z,e_{\be}],[h,e_{\al}]] \\
& = -(\al(z)\be(h)-\be(z)\al(h))[e_{\al},e_{\be}] \\
& = -n_{\al\be}(\al(z)\be(h)-\be(z)\al(h))e_{\al+\be}.
\end{align*}
Applying $\ad(z)$ on the left-hand side we obtain by \eqref{3},
\begin{align*}
0 & = n_{\al\be}(\al+\be)(z)(\al(z)\be(h)-\be(z)\al(h))
\end{align*}
for all $h\in \Lh$. This means  $(\al(z)\be(h)-\be(z)\al(h))=0$ for all $h\in \Lh$. This implies
$\al(z)=\be(z)=0$, so that $(\al+\be)(z)=0$, a contradiction. \\
So we may assume that $z$ is nilpotent. By Morozov's theorem $[z,x]=z$ for some $x\in \Lg$, so that
identity \eqref{4} implies $\ad (z)^3=0$. Again by the limit argument we can assume that $z$ lies in the
minimal nilpotent orbit, i.e., $z=e_{\theta}$, where $\theta$ is the maximal root, or just a long root. 
This implies that we may already assume that $\Lg$ is of type
$A_2$, $B_2$ or $G_2$. But now a direct computation with a computer algebra system shows that in all
three cases the identity forces $z=0$ and we are done.    
\end{proof}

\begin{cor}
Let $\Lg$ be a simple Lie algebra of rank $r\ge 2$ over an algebraically closed field $K$ of characteristic
zero, and suppose that $R=\ad (z)$ satisfies the modified Yang-Baxter equation, that is 
\begin{align}\label{mybe}
[z,[z,[x,y]]]=[[z,x],[z,y]]+\la [x,y]
\end{align}
for all $x,y\in \Lg$. Then $z=0$ and $\la=0$.
\end{cor}

Note that the operator form of identity \eqref{mybe} is given by
\[
\ad(z)^2\ad(x)-\ad(z)\ad(x)\ad(z)+\ad(x)\ad(z)^2=\la \ad (x)
\]
for all $x$. For the rank one case the following result can be shown, for all fields
$K$ of characteristic zero, again by a direct computation.

\begin{prop}
Let $\Lg=\Ls\Ll(2,K)$ with standard basis $(e_1,e_2,e_3)$, and $z=z_1e_1+z_2e_2+z_3e_3$. 
Then $R=ad(z)$ solves MYBE for $z\in \Lg$ and $\la\in K$ if and only if $\la=4(z_1z_2+z_3^2)$.
\end{prop}

Concerning identity \eqref{2} we obtain a result analogous to Theorem $\ref{3.3}$, but for
all simple Lie algebras.

\begin{prop}
Let $\Lg$ be a simple Lie algebra over an algebraically closed field $K$ of characteristic
zero, and $D$ be a derivation of $\Lg$ satisfying the identity
\eqref{2}. Then $D=0$.
\end{prop}

\begin{proof}
The result follows from Theorem $\ref{3.3}$ if $\Lg$ has rank $r\ge 2$, and for the rank $1$ case
by a direct computation. On the other hand, there is also a direct proof without a computation.
Since every derivation of $\Lg$ is inner, there is an element $z\in \Lg$ with $D=\ad (z)$.
Then identity \eqref{2} gives, with $w=z$, 
\[
[[z,x],[z,y]]=0
\]
for all $x,y\in \Lg$. As before in Theorem $\ref{3.3}$ we may assume that $z$ is nilpotent.
For $z=0$ we are done. Otherwise, since $z$ is nilpotent, there
exists an element $x\in \Lg$ with $[z,x]=z$ by Jacobson-Morozov.
This implies $[z,[z,y]]=0$ for all $y\in \Lg$, so that $z$ is a sandwich element, i.e., with $\ad(z)^2=0$. 
It is well-known that there are no nontrivial sandwich elements in simple Lie algebras. Hence $z=0$ and $D=0$.
\end{proof}

We note that for nilpotent Lie algebras we have a quite different behavior. There we always
have a nontrivial solution for \eqref{1} and \eqref{2}, different from the trivial derivation $D=0$.

\begin{prop}
Let $\Lg$ be nilpotent. Then there exists a nontrivial derivation $D\in \Der(\Lg)$ 
satisfying identities \eqref{1} and \eqref{2}.
\end{prop}

\begin{proof}
If $\Lg$ is nilpotent of class $c(\Lg)\le 2$, then obviously identity \eqref{2}, and hence also \eqref{1},
holds for every derivation. Consider the lower central series of $\Lg$. Suppose that $\Lg^k=0$ and $\Lg^{k-1}\neq 0$,
with $k\ge 3$, i.e., $c(\Lg)\ge 3$. Choose an element $w$ in $\Lg^{k-2}$ which is not in the center of 
$\Lg$. This is possible, because otherwise $\Lg^{k-2}\subseteq Z(\Lg)$, and hence $\Lg^{k-1}=[\Lg,\Lg^{k-2}]=0$, 
a contradiction. It follows that $D=\ad(w)$ is a nontrivial derivation satisfying identity \eqref{2}, because 
each term is zero. Indeed, $[[w,x],[y,z]]\in [\Lg^{k-1},\Lg^1]\subseteq \Lg^k=0$.
\end{proof}

\section{Lie algebra identities}

In this section we study Lie algebras $\Lg$ satisfying one of the
identities \eqref{1}, \eqref{2}, \eqref{3}, \eqref{4}, i.e., 
\begin{align*}
D\bigl([D(x),[y,z]]+[D(y),[z,x]]+[D(z),[x,y]] \bigr) & =0, \\
[D(x),[y,z]]+[D(y),[z,x]]+[D(z),[x,y]] & =0, \\
[z,[[z,x],[y,w]]]+[z,[[z,y],[w,x]]]+[z,[[z,w],[x,y]]] & =0, \\
[z,[[z,x],[z,y]]] & = 0,
\end{align*}
for all $x,y,z,w\in \Lg$ and not just for a given derivation, but 
{\it for all} derivations $D\in \Der(\Lg)$. 
This leads us to the theory of Lie algebra identities, which has a large literature.
Clearly we have the implications $ \eqref{2}\Rightarrow  \eqref{1} \Rightarrow  \eqref{3} \Rightarrow  \eqref{4}$.
Concerning question $\ref{2}$ , identity \eqref{1} holds 
if and only if $[x,y]_D=D([x,y])$ is a  Lie bracket for all derivations $D\in \Der(\Lg)$. \\
We start with the identity $[[z,x],[z,y]]=0$, which is a consequence of \eqref{2} by
taking $D=\ad (z)$. A Lie algebra satisfying this identity is metabelian, i.e.,
satisfies $\Lg^{(2)}=0$. This is known but rarely mentioned. Therefore it seems useful to give a proof 
here. We always assume that the field $K$ has characteristic zero, if not said otherwise.

\begin{lem}\label{4.1}
Let $\Lg$ be a Lie algebra over a field $K$ of characteristic not $2$. Then $\Lg$ is metabelian
if and only if it satisfies the identity
\[
[[z,x],[z,y]]=0
\]
for all $x,y,z\in \Lg$.
\end{lem}

\begin{proof}
Suppose that $\Lg$ is metabelian, i.e., satisfies the identity $[[z,x],[w,y]]=0$. Setting $w=z$
we obtain the required identity. Conversely, if we assume $[[z,x],[z,y]]=0$ and formally replace $z$
by $u+v$ we obtain
\[
[[u,x],[v,y]]=[[u,y],[v,x]]
\]
for all $x,y,u,v$. Now we use this identity and skew-symmetry twice to obtain
\begin{align*}
[[z,x],[w,y]] & = [[w,y],[x,z]] \\
              & = [[w,z],[x,y]] \\
              & = [[z,w],[y,x]] \\
              & = [[z,x],[y,w]].
\end{align*}
This implies $2[[z,x],[w,y]]=0$.
\end{proof}

For Lie algebras satisfying the strongest identity, namely \eqref{2}, we obtain the following 
necessary condition.
 
\begin{prop}
Let $\Lg$ be a Lie algebra satisfying identity \eqref{2}. Then $\Lg$ is metabelian.
\end{prop}

\begin{proof}
Applying identity \eqref{2} for $D=\ad(w)$ we obtain
\[
[[w,x],[y,z]]+[[w,y],[z,x]]+[[w,z],[x,y]]=0
\]
for all $x,y,z,w\in \Lg$. Setting $w=z$ this implies
\[
[[z,x],[z,y]]=0
\]
for all $x,y,z\in \Lg$. By Lemma $\ref{4.1}$, $\Lg$ is metabelian.
\end{proof}

The converse is not true in general, but of course true for all inner derivations.

\begin{cor}
A Lie algebra $\Lg$ satisfies identity \eqref{2} for all inner derivations if and only if
it is metabelian.
\end{cor}

The problem for the converse in general are the outer derivations of a metabelian Lie algebra.
They need not satisfy identity \eqref{2}. The following example demonstrates this, and is of minimal 
dimension with this property.

\begin{ex}\label{4.4}
Let $\Lg$ be the non-nilpotent metabelian Lie algebra of dimension $4$ with
basis $\{e_1,\ldots ,e_4\}$, defined by the brackets
\[
[e_1,e_2]=e_2,\; [e_1,e_3]=e_2,\; [e_1,e_4]=e_4,\; [e_2,e_3]=e_4.
\]
Then the outer derivations $D=\diag (0,\la,\la,2\la)$ satisfy 
\eqref{2} if and only if $\la=0$. They also satisfy  \eqref{1} if and only if $\la=0$.
\end{ex}

Indeed, we have 
\[
[D(e_1),[e_2,e_3]]+[D(e_2),[e_3,e_1]]+ [D(e_3),[e_1,e_2]]  = -2\la e_4.
\]

There are also sufficient conditions for a Lie algebra $\Lg$ to satisfy identity \eqref{2}, such
as $\Lg^2=[\Lg,[\Lg,\Lg]]=0$. One would like to find more interesting conditions, of course.
A view on low-dimensional Lie algebras already shows that it is not so easy.

\begin{prop}
All complex Lie algebras of dimension $n\le 4$ satisfy identity \eqref{2} with the exception
of $\Ls\Ll_2(\C)$, $\Lg\Ll_2(\C)$, $\Lg_3$ and $\Lg_5(\al)$, which are listed in table $1$.
\end{prop}

Note that the algebra given in Example $\ref{4.4}$ is $\Lg_5(0)$. \\[0.2cm]
A finite-dimensional Lie algebra $\Lg$ is called {\it almost abelian},
if it has an abelian ideal $\La$ of codimension $1$. We may choose a basis $\{e_1,\ldots ,e_n \}$ of $\Lg$ such that
$\La=\langle e_2,\ldots ,e_n\rangle$ and $\Lg \simeq \La \rtimes \langle e_1\rangle $. 

\begin{prop}\label{4.6}
Any almost abelian Lie algebra satisfies identity \eqref{2}.
\end{prop}

\begin{proof}
Let $\Lg=\La \rtimes \langle e_1\rangle$ be an almost abelian Lie algebra of dimension $n$. 
Since $\La$ is an ideal of codimension $1$ we know that $[\Lg,\Lg]\subseteq \La$. Consider the 
annihilator of $[\Lg,\Lg]$ in $\Lg$,
\[
\Lh=\{x\in \Lg \mid [x,[\Lg,\Lg]]=0 \}.
\]
Since $[\Lg,\Lg]$ is a characteristic ideal of $\Lg$, i.e., $D([\Lg,\Lg])\subseteq [\Lg,\Lg]$ for all
derivations $D$ of $\Lg$, we have
\begin{align*}
[D(\Lh),[\Lg,\Lg]] & = D([\Lh,[\Lg,\Lg]])+[\Lh,D([\Lg,\Lg])] \\
                   & \subseteq 0 + [\Lh,[\Lg,\Lg]].
\end{align*}
This shows $D(\Lh)\subseteq \Lh$ for all $D\in \Der(\Lg)$. Taking $D=\ad(x)$ we see that $\Lh$ is an ideal
of $\Lg$. So $\Lh$ is a characteristic ideal of $\Lg$ with $\La\subseteq \Lh$. We have $\Lh=\Lg$ if and
only if $c(\Lg)\le 2$. However, for Lie algebras of nilpotency class at most $2$ we
are done. Otherwise we have $\Lh=\La$, i.e., $D(\La)\subseteq \La$ for all $D\in \Der(\Lg)$.
Now it is easy to see that the identity
\[
[D(e_i),[e_j,e_k]]+[D(e_j),[e_k,e_i]]+[D(e_k),[e_i,e_j]]=0
\]
holds for all $i,j,k$. If all $i,j,k\ge 2$, then all $[e_r,e_s]\in [\La,\La]=0$ for $r,s\in \{i,j,k \}$.
Hence we may assume that $i=1$. If both $j,k\ge 2$, then
\[
[D(e_1),[e_j,e_k]]=[D(e_j),[e_k,e_1]]=[D(e_k),[e_1,e_j]] = 0, 
\]
since $[e_j,e_k]=0$ and $D(e_j),D(e_k)\in \La$ because of $D(\La)\subseteq \La$.
Hence we may assume that $i=j=1$. But then we obtain
\[
[D(e_1),[e_1,e_k]]+[D(e_1),[e_k,e_1]]+[D(e_k),[e_1,e_1]]=0,
\] 
so that identity \eqref{2} is satisfied. \\
\end{proof}

An example for an almost abelian Lie algebra is the standard graded filiform Lie algebra
$\Lf_n$ of dimension $n\ge 3$, with basis $\{ e_1,\ldots ,e_n\}$ and Lie brackets 
$[e_1,e_i]=e_{i+1}$ for $i=2,\ldots ,n-1$.

\begin{cor}
The filiform nilpotent Lie algebra $\Lf_n$ satisfies identity \eqref{2} for every $n\ge 3$.
\end{cor}

\begin{rem}
The above result shows that the nilpotency class of a Lie algebra satisfying identity \eqref{2}
can be arbitrarily large, whereas the solvability class is bounded by $2$.
\end{rem}

It is easy to verify the following result for low-dimensional nilpotent Lie algebras.
The notation is taken from Magnin \cite{MAG}.

\begin{prop}
Every complex nilpotent Lie algebra of dimension $n\le 5$ satisfies identity \eqref{2}.
In dimension $6$ all nilpotent algebras satisfy \eqref{2} with the exception of
$\Lg_{6,9}$, $\Lg_{6,13}$, $\Lg_{6,15}$, $\Lg_{6,18}$, $\Lg_{6,19}$, and $\Lg_{6,20}$.
\end{prop}

One can obtain a similar result for dimension $7$ by using
the classification list of Magnin - see table $2$. We have shortened Magnin's
notation there by omitting the dimension index $7$. \\
There is one interesting infinite family
of Lie algebras, depending on a complex parameter $\lambda$, where identity \eqref{2}
holds precisely for one singular value of $\la$. The family is $\Lg_{7,1.2(i_{\la})}$ in Magnin's notation:

\begin{ex}
For $\la\in \C$ let $\Lg_{\la}$ denote the following complex $7$-dimensional nilpotent Lie algebra given by the Lie brackets
\begin{align*}
[x_1,x_2] & = x_4,\; [x_1,x_3]=x_6,\; [x_1,x_4]=x_5,\; [x_1,x_5]=x_7, \\
[x_2,x_3] & = \la x_5,\; [x_2,x_4]=x_6,\; [x_2,x_6]=x_7,\; [x_3,x_4]=(1-\la)x_7. \\
\end{align*}
Then $(\ref{1}) \text{ holds }\Leftrightarrow (\ref{2})\text{ holds } \Leftrightarrow \la=1$.
\end{ex}

Note that we have $c(\Lg_{\la})=4$, $d(\Lg_{\la})=2$ for all $\la\in \C$, and
\[
\dim {\rm Der}(\Lg_{\la})=\begin{cases} 13, \text{ for } \la=-1 \\ 12, \text{ for } \la\neq -1 \end{cases}.
\]
One might ask for invariants which differ exactly for $\la=1$ and $\la\neq 1$. 
For the $(t,1,1)$-space of generalized derivations with $t\neq 0,1,-1,2$ we have
\[
\dim {\rm Der}_{(t,1,1)}(\Lg_{\la})=\begin{cases} 12, \text{ for } \la=1 \\ 11, \text{ for } \la\neq 1 \end{cases},
\]
see \cite{BU44}, \cite{ZUS} for more on generalized derivations. However, there seems to be no relation in general 
between these spaces and identity \eqref{1} or \eqref{2}. \\
Identity \eqref{1} is weaker than identity \eqref{2} in general. Indeed, a Lie algebra satisfying identity
\eqref{1} need not be metabelian as we have already seen in the cases of $\Ls\Ll_2(\C)$ and $\Lg\Ll_2(\C)$, see table $1$.
However, for nilpotent Lie algebras of dimension $n\le 6$ they are equivalent, and also
for nilpotent Lie algebras of dimension $7$ which are not CNLAs. 
This follows from an easy but lengthy computation, see table $2$.

\begin{prop}
A complex nilpotent Lie algebra of dimension $n\le 6$ satisfies identity
\eqref{1} if and only it satisfies identity \eqref{2}. In dimension $7$ this holds true
for all complex nilpotent Lie algebras which are not a CNLA.  
\end{prop}

For CNLAs of dimension $7$ identity \eqref{1} always holds, but identity \eqref{2} does not.

\begin{prop}\label{4.12}
Every complex CNLA of dimension $7$ satisfies identity \eqref{1}.
\end{prop}

In general, the last result is not true in higher dimension.

\begin{ex}\label{4.13}
Let $\Lg$ be the filiform nilpotent Lie algebra of dimension $8$ defined by the brackets
\begin{align*}
[x_1,x_i] & = x_{i+1},\; i=2,\ldots ,7, \\
[x_2,x_3] & = x_5+x_6,\; [x_2,x_4]=x_6+x_7,\; [x_2,x_5]=2x_7+x_8, \\
[x_2,x_6] & = 3 x_8,\; [x_3,x_4]=-x_7,\; [x_3,x_5]=-x_8. 
\end{align*}
Then $\Lg$ is a CLNA which does not satisfy identity \eqref{1}.
\end{ex}

In fact, $\Lg$ does not even satisfy identity \eqref{4}, since we have
\[
[x_1,[[x_1,x_2],[x_1,x_3]]]=-x_8.
\]
We have $c(\Lg)=7$ and $d(\Lg)=3$. \\[0.2cm]
Let us finally discuss the identities \eqref{3} and \eqref{4}. They have been studied
by many authors in connection with identities in $\Ls\Ll_2(K)$. Identity \eqref{4} appears
in the basis for identities of $\Ls\Ll_2(K)$ found by Razmyslov \cite{RAZ}:

\begin{thm}
Let $K$ be a field of characteristic zero. A finite basis of identities for the Lie
algebra $\Ls\Ll_2(K)$ is given by identity \eqref{4} and the standard identity of degree $5$,
\[
\sum_{\pi\in S_4}(-1)^{\pi}[x_{\pi(1)}, [x_{\pi(2)},[x_{\pi(3)},[x_{\pi(4)},x_0]]]]=0
\]
for all $x_i\in \Ls\Ll_2(K)$.
\end{thm}

This theorem was generalized by Fillipov  \cite{FIL1} to arbitrary fields $K$ of characteristic not $2$.
Moreover he showed that all such identities for $\Ls\Ll_2(K)$ are a consequence of one
single identity, namely 

\begin{align}\label{6}
[z,[[w,x],[w,y]]] & = [w,[[z,w],[x,y]]].
\end{align} 

This is related to our identity \eqref{3} as follows.

\begin{prop}
Identity \eqref{6} is a consequence of identity \eqref{3}.
\end{prop}

\begin{proof}
Formally replacing $z$ by $z+v$ in \eqref{3} gives
\begin{align*}
0 & = [z+v,[[z+v,x],[y,w]]] + [z+v,[[z+v,y],[w,x]]] + [z+v,[[z+v,w],[x,y]]] \\
  & = [z,[[z,x],[y,w]]]+ [z,[[v,x],[y,w]]]+ [v,[[z,x],[y,w]]]+[v,[[v,x],[y,w]]] \\
  & + [z,[[z,y],[w,x]]]+ [z,[[v,y],[w,x]]]+ [v,[[z,y],[w,x]]]+[v,[[v,y],[w,x]]] \\
  & + [z,[[z,w],[x,y]]]+ [z,[[v,w],[x,y]]]+ [v,[[z,w],[x,y]]]+[v,[[v,w],[x,y]]]. 
\end{align*}
Applying \eqref{3} for the first and last column of terms we obtain
\begin{align*}
0 & = [z,[[v,x],[y,w]]] + [z,[[v,y],[w,x]]] + [z,[[v,w],[x,y]]] \\
  & + [v,[[z,x],[y,w]]] + [v,[[z,y],[w,x]]] + [v,[[z,w],[x,y]]].
\end{align*}
Setting $v=w$ and applying \eqref{3} with $z$ and $w$ interchanged, i.e., in the form 
$$
[w,[[w,x],[y,z]]]+[w,[[w,y],[z,x]]=[w,[z,w],[x,y]]],
$$
we obtain 
\begin{align*}
0 & = 2[z,[[w,x],[y,w]]]+ [w,[[z,x],[y,w]]] \\
  & + [w,[[z,y],[w,x]]]+[w,[[z,w],[x,y]] \\
  & = 2[z,[[w,x],[y,w]]]+2[w,[[z,w],[x,y]]].
\end{align*}
This is identity \eqref{6}.
\end{proof}

Identity \eqref{4} has been studied further, but mostly for simple and semisimple algebras. 
Filippov \cite{FIL2} termed algebras satisfying identity \eqref{4} also {\it $h_0$-algebras}, and algebras satisfying
identity \eqref{6} also {\it $h$-algebras}. Another term for the variety of $h$-algebras is given
by ${\rm var}(\Ls\Ll_2(K))$. This variety and its subvarieties have also been studied by several
authors, see \cite{DRE} and the references therein. \\
A study of identities \eqref{3} and \eqref{4} for solvable and nilpotent Lie algebras seems to be less known.
Table $1$ shows the result for complex Lie algebras of dimension $n\le 4$.

\begin{prop}
All complex Lie algebras of dimension $n\le 4$ satisfy identity \eqref{3}, and hence also \eqref{4}, except
for $\Lg_3$ and $\Lg_5(\al)$ with $\al\neq 0,-1$.
\end{prop}

There is a trivial reason in low dimensions, why these identities are often satisfied. Every center-by-metabelian
Lie algebra $\Lg$ satisfies identity  \eqref{3} and  \eqref{4}. By definition, center-by-metabelian means
that $\Lg^{(2)}\subseteq Z(\Lg)$. This immediately implies that every term in \eqref{3} is zero.
Indeed, all low-dimensional nilpotent Lie algebras are  center-by-metabelian. More precisely, we have the following
result:

\begin{prop}
Every nilpotent Lie algebra $\Lg$ of dimension $n\le 7$ over a field of characteristic zero
is center-by-metabelian, and hence satisfies identity \eqref{3} and \eqref{4}.
\end{prop}

\begin{proof}
The claim follows from results in \cite{BOK}. We have $\Lg^{(3)}=0$ because of $n\le 7$.
Suppose that $\Lg^{(2)}\neq 0$. Then  $\dim \Lg/\dim \Lg^{(1)}\ge 2$,
$\dim \Lg^{(1)}/\dim \Lg^{(2)}\ge 3$, and $n\ge 6$. Moreover, if $\dim \Lg^{(1)}/\dim \Lg^{(2)}= 3$, then
$\dim \Lg^{(2)}\le 1$, see \cite{BOK}. Otherwise we have $\dim \Lg^{(1)}/\dim \Lg^{(2)}\ge 4$ and
\begin{align*}
n & =\dim \Lg/\dim \Lg^{(1)} + \dim \Lg^{(1)}/\dim \Lg^{(2)} + \dim \Lg^{(2)} \\
  & \ge 6+  \dim \Lg^{(2)}
\end{align*}
This gives again $\dim \Lg^{(2)}\le 1$.
Because $\Lg$ is nilpotent, $\Lg^{(2)}\cap Z(\Lg)\neq 0$, and hence $\Lg^{(2)}\subseteq Z(\Lg)$.
\end{proof}

The result does not hold in higher dimensions. Indeed, the
$8$-dimensional nilpotent Lie algebra of Example $\ref{4.13}$ is not center-by-metabelian.
It does not satisfy identity \eqref{3} or \eqref{4}. Recall that  a Lie algebra $\Lg$ satisfying 
identity \eqref{3}, or \eqref{4} need not be center-by-metabelian, e.g., consider $\Ls\Ll_2(K)$.

\section{Tables}

1.) Complex Lie algebras of dimension $n\le 4$:
\vspace*{0.5cm}
\begin{center}
\begin{tabular}{c|c|c|c|c|c}
$\Lg$ & Lie brackets & \eqref{1} & \eqref{2} & \eqref{3} & \eqref{4} \\
\hline
$\Lr_2(\C)$ & $[e_1,e_2]=e_2$ & $\checkmark$ & $\checkmark$ & $\checkmark$ & $\checkmark$ \\
$\Ln_3(\C)$ & $[e_1,e_2]=e_3$  & $\checkmark$ & $\checkmark$ & $\checkmark$ & $\checkmark$ \\
$\Lr_{3,\la}(\C)$ & $[e_1,e_2]=e_2,\, [e_1,e_3]=\la e_3$  & $\checkmark$ & $\checkmark$ & $\checkmark$ & $\checkmark$ \\
$\Ls\Ll_2(\C)$  &  $[e_1,e_2]=e_3,\, [e_1,e_3]= -2e_1,\,[e_2,e_3]= 2e_2$ & $\checkmark$ & $-$ & $\checkmark$ & $\checkmark$ \\
$\Ln_3(\C)\oplus \C$ & $[e_1,e_2]=e_3$ & $\checkmark$ & $\checkmark$ & $\checkmark$ & $\checkmark$ \\
$\Ln_4(\C)$ & $[e_1,e_2]=e_3$,\, $[e_1,e_3]=e_4$ & $\checkmark$ & $\checkmark$ & $\checkmark$ & $\checkmark$ \\
$\Lr_2(\C)\oplus \C^2$ & $[e_1,e_2]=e_2$ & $\checkmark$ & $\checkmark$ & $\checkmark$ & $\checkmark$ \\
$\Lr_2(\C)\oplus \Lr_2(C)$ & $[e_1,e_2]=e_2,\, [e_3,e_4]=e_4$  & $\checkmark$ & $\checkmark$ & $\checkmark$ & $\checkmark$ \\
$\Ls\Ll_2(\C)\oplus \C$ & $[e_1,e_2]=e_3,\, [e_1,e_3]= -2e_1,\,[e_2,e_3]= 2e_2 $ & $\checkmark$ & $-$ 
                        & $\checkmark$ & $\checkmark$ \\
$\Lg_1$ & $[e_1,e_2]=e_2, \,[e_1,e_3]=e_3,\,[e_1,e_4]=e_4$ & $\checkmark$ & $\checkmark$ & $\checkmark$ & $\checkmark$ \\
$\Lg_2(\al)$ & $[e_1,e_2]=e_2, \,[e_1,e_3]=e_3,\,[e_1,e_4]=e_3+\al e_4$ & $\checkmark$ & $\checkmark$ & $\checkmark$ 
             & $\checkmark$ \\
$\Lg_3$ & $[e_1,e_2]=e_2, \,[e_1,e_3]=e_3,\,[e_1,e_4]=2e_4,\,[e_2,e_3]=e_4$ & $-$ & $-$ & $-$ & $-$ \\
$\Lg_4(\al,\be)$ & $[e_1,e_2]=e_2, \,[e_1,e_3]=e_2+\al e_3,\,[e_1,e_4]=e_3+\be e_4$ & $\checkmark$ & $\checkmark$ 
                 & $\checkmark$ & $\checkmark$ \\
$\Lg_5(\al),\, \al\neq 0,-1$ & $[e_1,e_2]=e_2, \,[e_1,e_3]=e_2+\al e_3,$ & $-$ & $-$ & $-$ & $-$ \\
           & $[e_1,e_4]=(\al +1)e_4,\,[e_2,e_3]=e_4$ & & & & \\
$\al=0,-1$ & & $-$ & $-$ & $\checkmark$ & $\checkmark$ \\
\end{tabular}
\end{center}
\vspace*{0.5cm}
2.) Indecomposable complex nilpotent Lie algebras of dimension $7$, see \cite{MAG}: 
\vspace*{0.5cm}
\begin{center}
\begin{tabular}{c|c|c|c|c|c|c|c|c|c|c|c|c}
Identity & $\Lg_{0.1}$ & $\Lg_{0.2}$ & $\Lg_{0.3}$ & $\Lg_{0.4(\lambda)}$ & $\Lg_{0.5}$ & $\Lg_{0.6}$ & $\Lg_{0.7}$ 
& $\Lg_{0.8}$ & $\Lg_{1.01(i)}$ & $\Lg_{1.01(ii)}$  & $\Lg_{1.02}$ & $\Lg_{1.03}$ \\
\hline
\eqref{1} & $\checkmark$  & $\checkmark$  & $\checkmark$  & $\checkmark$  & $\checkmark$ & $\checkmark$ 
& $\checkmark$ & $\checkmark$ & $\checkmark$ & $\checkmark$ & $-$ & $-$ \\
\eqref{2} & $-$  & $\checkmark$  & $\checkmark$  & $-$  & $-$ & $-$ 
& $-$ & $-$ & $\checkmark$ & $\checkmark$ & $-$ & $-$ \\
\end{tabular} 
\end{center}
\vspace*{0.5cm}
\begin{center}
\begin{tabular}{c|c|c|c|c|c|c|c|c|c}
Identity & $\Lg_{1.1(i_{\la})}$ & $\Lg_{1.1(ii)}$ & $\Lg_{1.1(iii)}$ & $\Lg_{1.1(iv)}$ & $\Lg_{1.1(v)}$ & $\Lg_{1.1(vi)}$ 
& $\Lg_{1.2(i_{\la\neq 1})}$ & $\Lg_{1.2(i_{\la=1})}$ & $\Lg_{1.2(ii)}$ \\
\hline
\eqref{1} & $-$  & $-$ & $-$ & $-$ & $-$ & $-$ &  $-$ & $\checkmark$  & $-$ \\
\eqref{2} & $-$  & $-$ & $-$ & $-$ & $-$ & $-$ &  $-$ & $\checkmark$  & $-$ \\
\end{tabular} 
\end{center}
\vspace*{0.5cm}
\begin{center}
\begin{tabular}{c|c|c|c|c|c|c|c|c|c|c|c}
Identity & $\Lg_{1.2(iii)}$ & $\Lg_{1.2(iv)}$ & $\Lg_{1.3(i_{\la})}$ & $\Lg_{1.3(ii)}$ & $\Lg_{1.3(iii)}$ & $\Lg_{1.3(iv)}$ 
& $\Lg_{1.3(v)}$ & $\Lg_{1.4}$ & $\Lg_{1.5}$ & $\Lg_{1.6}$ & $\Lg_{1.7}$ \\
\hline
\eqref{1} & $-$  & $-$  & $-$  & $-$ & $\checkmark$ & $\checkmark$ & $-$ & $-$ & $-$ & $\checkmark$ &  $\checkmark$  \\
\eqref{2} & $-$  & $-$  & $-$  & $-$ & $\checkmark$ & $\checkmark$ & $-$ & $-$ & $-$ & $\checkmark$ &  $\checkmark$  \\
\end{tabular} 
\end{center}
\vspace*{0.5cm}
\begin{center}
\begin{tabular}{c|c|c|c|c|c|c|c|c|c|c|c|c}
Identity & $\Lg_{1.8}$ & $\Lg_{1.9}$ & $\Lg_{1.10}$ & $\Lg_{1.11}$ & $\Lg_{1.12}$ & $\Lg_{1.13}$ 
& $\Lg_{1.14}$ & $\Lg_{1.15}$ & $\Lg_{1.16}$ & $\Lg_{1.17}$ & $\Lg_{1.18}$  & $\Lg_{1.19}$ \\
\hline
\eqref{1} & $\checkmark$  & $\checkmark$  & $\checkmark$  & $-$  & $-$ &  $-$ & $-$ & $\checkmark$ & $\checkmark$ 
& $-$ & $\checkmark$ & $\checkmark$ \\
\eqref{2} & $\checkmark$  & $\checkmark$  & $\checkmark$  & $-$  & $-$ &  $-$ & $-$ & $\checkmark$ & $\checkmark$ 
& $-$ & $\checkmark$ & $\checkmark$ \\
\end{tabular} 
\end{center}
\vspace*{0.5cm}
\begin{center}
\begin{tabular}{c|c|c|c|c|c|c|c|c|c|c|c|c}
Identity & $\Lg_{1.20}$ & $\Lg_{1.21}$ & $\Lg_{2.1(i_{\la})}$ & $\Lg_{2.1(ii)}$ & $\Lg_{2.1(iii)}$ & $\Lg_{2.1(iv)}$ 
& $\Lg_{2.1(v)}$ & $\Lg_{2.2}$ & $\Lg_{2.3}$ & $\Lg_{2.4}$ & $\Lg_{2.5}$ & $\Lg_{2.6}$ \\
\hline
\eqref{1} & $-$  & $-$  & $-$  & $-$  & $-$ &  $-$ & $-$ & $-$ & $\checkmark$ & $-$ & $-$ & $-$ \\
\eqref{2} & $-$  & $-$  & $-$  & $-$  & $-$ &  $-$ & $-$ & $-$ & $\checkmark$ & $-$ & $-$ & $-$ \\
\end{tabular} 
\end{center}
\vspace*{0.5cm}
\begin{center}
\begin{tabular}{c|c|c|c|c|c|c|c|c|c|c|c|c}
Identity & $\Lg_{2.7}$ & $\Lg_{2.8}$ & $\Lg_{2.9}$ & $\Lg_{2.10}$ & $\Lg_{2.11}$ & $\Lg_{2.12}$ 
& $\Lg_{2.13}$ & $\Lg_{2.14}$ & $\Lg_{2.15}$ & $\Lg_{2.16}$ & $\Lg_{2.17}$ & $\Lg_{2.18}$ \\
\hline
\eqref{1} & $\checkmark$ & $-$  & $\checkmark$ & $-$  & $\checkmark$ & $\checkmark$ & $-$ & $-$ & $-$ & $\checkmark$ 
& $-$ & $\checkmark$ \\
\eqref{2} & $\checkmark$ & $-$  & $\checkmark$ & $-$  & $\checkmark$ & $\checkmark$ & $-$ & $-$ & $-$ & $\checkmark$ 
& $-$ & $\checkmark$ \\
\end{tabular} 
\end{center}
\vspace*{0.5cm}
\begin{center}
\begin{tabular}{c|c|c|c|c|c|c|c|c|c|c|c|c}
Identity & $\Lg_{2.19}$ & $\Lg_{2.20}$ & $\Lg_{2.21}$ & $\Lg_{2.22}$ & $\Lg_{2.23}$ & $\Lg_{2.24}$ 
& $\Lg_{2.25}$ & $\Lg_{2.26}$ & $\Lg_{2.27}$ & $\Lg_{2.28}$ & $\Lg_{2.29}$ & $\Lg_{2.30}$ \\
\hline
\eqref{1} & $\checkmark$ & $\checkmark$ & $\checkmark$ & $\checkmark$ & $\checkmark$ & $-$ & $-$ & $-$ & $\checkmark$ 
& $\checkmark$ & $\checkmark$ & $\checkmark$ \\
\eqref{2} & $\checkmark$ & $\checkmark$ & $\checkmark$ & $\checkmark$ & $\checkmark$ & $-$ & $-$ & $-$ & $\checkmark$ 
& $\checkmark$ & $\checkmark$ & $\checkmark$ \\
\end{tabular} 
\end{center}
\vspace*{0.5cm}
\begin{center}
\begin{tabular}{c|c|c|c|c|c|c|c|c|c|c|c|c}
Identity & $\Lg_{2.31}$ & $\Lg_{2.32}$ & $\Lg_{2.33}$ & $\Lg_{2.34}$ & $\Lg_{2.35}$ & $\Lg_{2.36}$ 
& $\Lg_{2.37}$ & $\Lg_{2.38}$ & $\Lg_{2.39}$ & $\Lg_{2.40}$ & $\Lg_{2.41}$ & $\Lg_{2.42}$ \\
\hline
\eqref{1} & $\checkmark$ & $\checkmark$ & $\checkmark$ & $\checkmark$ & $-$ & $\checkmark$ & $-$ & $\checkmark$ & $\checkmark$ 
& $\checkmark$ & $-$ & $\checkmark$ \\
\eqref{2} & $\checkmark$ & $\checkmark$ & $\checkmark$ & $\checkmark$ & $-$ & $\checkmark$ & $-$ & $\checkmark$ & $\checkmark$ 
& $\checkmark$ & $-$ & $\checkmark$ \\
\end{tabular} 
\end{center}
\vspace*{0.5cm}
\begin{center}
\begin{tabular}{c|c|c|c|c|c|c|c|c|c|c|c|c}
Identity & $\Lg_{2.43}$ & $\Lg_{2.44}$ & $\Lg_{2.45}$ & $\Lg_{3.1(i_{\la})}$ & $\Lg_{3.1(iii)}$ & $\Lg_{3.2}$ 
& $\Lg_{3.3}$ & $\Lg_{3.4}$ & $\Lg_{3.5}$ & $\Lg_{3.6}$ & $\Lg_{3.7}$ & $\Lg_{3.8}$ \\
\hline
\eqref{1} & $\checkmark$ & $\checkmark$ & $\checkmark$ & $-$ & $-$ & $\checkmark$ & $\checkmark$ & $\checkmark$ & $-$ 
& $\checkmark$ & $\checkmark$ & $\checkmark$ \\
\eqref{2} & $\checkmark$ & $\checkmark$ & $\checkmark$ & $-$ & $-$ & $\checkmark$ & $\checkmark$ & $\checkmark$ & $-$ 
& $\checkmark$ & $\checkmark$ & $\checkmark$ \\
\end{tabular} 
\end{center}
\vspace*{0.5cm}
\begin{center}
\begin{tabular}{c|c|c|c|c|c|c|c|c|c|c|c|c}
Identity & $\Lg_{3.9}$ & $\Lg_{3.10}$ & $\Lg_{3.11}$ & $\Lg_{3.12}$ & $\Lg_{3.13}$ & $\Lg_{3.14}$ 
& $\Lg_{3.15}$ & $\Lg_{3.16}$ & $\Lg_{3.17}$ & $\Lg_{3.18}$ & $\Lg_{3.19}$ & $\Lg_{3.20}$ \\
\hline
\eqref{1} & $\checkmark$ & $-$ & $\checkmark$ & $\checkmark$ & $\checkmark$ & $\checkmark$ & $\checkmark$ & $\checkmark$ 
& $\checkmark$ & $\checkmark$ & $\checkmark$ & $\checkmark$ \\
\eqref{2} & $\checkmark$ & $-$ & $\checkmark$ & $\checkmark$ & $\checkmark$ & $\checkmark$ & $\checkmark$ & $\checkmark$ 
& $\checkmark$ & $\checkmark$ & $\checkmark$ & $\checkmark$ \\
\end{tabular} 
\end{center}
\vspace*{0.5cm}
\begin{center}
\begin{tabular}{c|c|c|c|c|c|c|c|c}
Identity & $\Lg_{3.21}$ & $\Lg_{3.22}$ & $\Lg_{3.23}$ & $\Lg_{3.24}$ & $\Lg_{4.1}$ & $\Lg_{4.2}$ 
& $\Lg_{4.3}$ & $\Lg_{4.4}$ \\
\hline
\eqref{1} & $\checkmark$ & $-$ & $\checkmark$ & $\checkmark$ & $\checkmark$ & $\checkmark$ & $\checkmark$ & $\checkmark$ \\
\eqref{2} & $\checkmark$ & $-$ & $\checkmark$ & $\checkmark$ & $\checkmark$ & $\checkmark$ & $\checkmark$ & $\checkmark$ \\
\end{tabular} 
\end{center}

\section{Acknowledgment}

The author would like to thank the referee for helpful remarks.

\end{document}